\documentclass
{amsart}
\usepackage{times}
\usepackage{amsfonts}
\usepackage{amsmath}
\usepackage{amscd}
\usepackage{amssymb}
\usepackage{amsxtra}
\usepackage{latexsym}
\usepackage{amsthm}
\usepackage[bookmarks,colorlinks,pagebackref]{hyperref} 

\newcommand \Exactseq [3]{0\to {#1}\to {#2}\to {#3}\to 0}

\newcommand \ann[2]{\operatorname{Ann}_{#1}(#2)}

\newcommand \CM{Coh\-en-Mac\-au\-lay}

\renewcommand \hom [3]{\operatorname{Hom}_{#1}(#2,#3)} 
\newcommand \homo{homomorphism}
\newcommand \id{\mathfrak a}
\renewcommand\iff{if and only if}

\newcommand \inverse[2]{{#1^{-1}(#2)}}
\newcommand \iso{\cong}

\newcommand \maxim{\mathfrak m}
\newcommand \nat{\mathbb N}

\newcommand \op\operatorname
\newcommand \pol[2]{#1[#2]}
\newcommand \pow[2]{#1[[#2]]}

\newcommand \pr{\mathfrak p}
\newcommand \primary{\mathfrak g}
\newcommand \range [2]{#1,\dots,#2}
\let\sub\subseteq

\newcommand \zet{\mathbb Z}

\newcommand\en{\wedge}
\newcommand  \topc{weakly compressible}

\newcommand\val[1]{\op{val}(#1)}

\newcommand\aleq{\preceq}
\newcommand\binord[1]{\mathfrak o({#1})}
\newcommand\fl[1]{\mathfrak{#1}}
\newcommand\ndo[1]{\op{End}(#1)}

\newcommand\add{equilateral}

\newcommand\order[2]{\op{ord}_{#1}(#2)}
\newcommand\lc[2]{\op{H}^0_{#2}(#1)}
\newcommand\lcl[2]{e_{#1}(#2)}
\newcommand\ord{\mathbf O}

\newcommand\low[1]{\op{ord}(#1)}
\newcommand\Ssum{\bigoplus}
\newcommand\ssum{\oplus}

\newcommand \len[1]{\op{len}(#1)}
\newcommand \genlen[2]{\ell^{\text{gen}}_{#1}(#2)}
\newcommand \lenmod[2]{\op{len}_{#1}(#2)}

\newcommand \fcyc[2]{\texttt{cyc}_{#1}(#2)}

\newcommand  \sch{Grassmanian}
\newcommand  \hdim{height rank}

\newcommand  \hd[1]{\mathfrak l(#1)}
\newcommand  \cd[2]{\binord{\fcyc{#1}{#2}}}
\newcommand  \cohrk[1]{\op{coh}(#1)}
\newcommand\finlen[1]{\op {len}^{\text{fin}}(#1)}

\newcommand \gr[1]{\mathfrak {Gr}(#1)}
\newcommand \grass[2]{\mathfrak {Gr}_{#1}(#2)}

\theoremstyle{plain}

\newtheorem*{conjectuur*}{Conjecture}

\swapnumbers
\newtheorem{theorem}[subsection]{Theorem}
\newcommand\Thm[1]{Theorem~\ref{#1}}
\newtheorem{corollary}[subsection]{Corollary}
\newcommand\Cor[1]{Corollary~\ref{#1}}
\newtheorem{lemma}[subsection]{Lemma}
\newcommand\Lem[1]{Lemma~\ref{#1}}
\newtheorem{proposition}[subsection]{Proposition}
\newcommand\Prop[1]{Proposition~\ref{#1}}

\newtheorem*{maintheorem}{Main Theorem}

\newtheorem*{citedtheorem}{Theorem}

\theoremstyle{definition}

\newtheorem{example}[subsection]{Example}

\theoremstyle{remark}

\newtheorem{remark}[subsection]{Remark}
\newcommand\Rem[1]{Remark~\ref{#1}}

\swapnumbers

\title {Ordinal length and the canonical topology}
\author{Hans Schoutens}
\date\today
\address{Department of Mathematics\\
365 5th Avenue\\
the CUNY Graduate Center\\
New York, NY 10016, USA}

\begin{document}
\begin{abstract} 
We extend the classical length function to an ordinal-valued invariant on the class of all finite-dimensional Noetherian modules.
We show how to calculate this combinatorial invariant by means of the fundamental cycle of the  module, thus linking the lattice of submodules to homological properties of the module. Using this, we define on a module its canonical topology, in which every morphism is continuous.
\end{abstract}

\maketitle



\section{Introduction}
The purpose of this paper is to lay the foundations of a new, ordinal-valued invariant in Commutative Algebra: the 
 (\emph{ordinal}) \emph{length} $\len M$ of a Noetherian module $M$, measuring (by means of an ordinal) the longest descending chain of submodules in $M$. 
Recall that an \emph{ordinal} is an isomorphism class of a total well-order (=admitting the descending chain condition), and the class of ordinals, ordered by the initial segment relation, is again a well-order; each natural number is an ordinal by identifying it with a chain of that length; the order of $(\nat,\leq)$ is denoted $\omega$ and is equal to the supremum of all $n\in\nat$; the order-type of the  lexicographical order on $\nat^d$ is denoted $\omega^d$. Apart from the usual (non-commutative) sum $\alpha+\beta$, we also need the (commutative) \emph{shuffle} or \emph{natural} sum $\alpha\ssum\beta$ given by coefficient-wise addition in the Cantor normal form (for details, see \S\ref{s:Ord}). Since descending chains in a Noetherian module $M$ are well-ordered with respect to the \emph{reverse} inclusion, they are ordinals, and we define $\len M$ as the supremum of all  ordinals obtained  this way. 

 In \cite{SchSemAdd} we showed that, whereas length can no longer be additive on exact sequences, it is still semi-additive:
 
 \begin{citedtheorem}[Semi-additivity]
If $\Exactseq NMQ$ is an exact sequence of Noetherian $R$-modules, then 
\begin{equation}\label{eq:lensemadd}
  \len Q+\len N\leq \len M\leq \len Q\ssum \len N.
\end{equation}
Moreover, if the sequence is split, then the last inequality is an equality.
\end{citedtheorem}
 
  As an application, we obtain that a Noetherian ring has length $\omega^d$ \iff\ it is a domain of Krull dimension $d$. In this paper, we prove the remarkable fact that this combinatorial invariant can be described in terms of homological invariants: recall that  the \emph{local multiplicity} $\lcl \pr M$ at a prime ideal $\pr$ is defined as the length of the $\pr$-torsion submodule of $M_\pr$ (whence is non-zero \iff\ $\pr$ is an associated prime of $M$). 

\begin{maintheorem}
The length of $M$ is equal to $\Ssum_\pr \lcl \pr M\omega^{\op{dim}(R/\pr)}$.
\end{maintheorem}

 Combining this with semi-additivity, leads to many interesting applications (\cite{SchSemAdd,SchBinEndo,SchCond}). In this paper, we apply the theory to define a canonical topology on every module: the open submodules in this topology are precisely those that have the same length as $M$. We show that any morphism between modules is continuous in this topology. Any open submodule is essential, and in \cite{SchCond}, we will characterize those modules for which the converse also holds. In the last section, as an application of this material, we discuss the phenomenon of \emph{degradation}: how source and target of a module may force the morphism to become (almost) zero. For instance, we show

 \begin{citedtheorem} 
If an endomorphism on $M$ factors through a module that has no associated primes in common with $M$, then it must be nilpotent.
\end{citedtheorem}

 \section{Ordinals and ordinal length}\label{s:Ord}
A partial ordering is called a \emph{(partial) well-order} if it has the
descending chain condition, that is to say, any descending chain must
eventually be constant. A total order is a well-order  \iff\ every non-empty
subset has a minimal element.  
An \emph{ordinal} is then an equivalence class, up to an order-preserving  isomorphism, of a total well-order.  
 The set of ordinals is a transfinite extension of the set of natural numbers $\nat$, in which the usual induction is replaced by transfinite induction.  
 We say that $\alpha\leq\beta$ if $\alpha$ can be embedded as a total order in $\beta$. Any bounded subset of  ordinals has  then  an infimum and a supremum.  The (Cantor) sum $\alpha+\beta$ is the ordinal corresponding to the well-order on $\alpha\sqcup\beta$ obtained by letting any element of $\beta$ be larger than any element of $\alpha$. Thus,   $1+\omega$ is the same as $\omega$, whence in particular different from $\omega+1$. We will not need arbitrary ordinal multiplication, but only products of the form $n\alpha$ with $n\in\nat$, simply defined as the sum of $n$ copies of $\alpha$ (be aware that logicians would use the more awkward notation   $\alpha\cdot n$ for this).  The supremum of the $n\omega$ is denoted $\omega^2$ and is the order-type of the lexicographical ordering on $\nat^2$. The $\omega^d$ are similarly defined, and their supremum is denoted $\omega^\omega$.
 
 Let $\ord$ be the collection of ordinals  strictly below $\omega^\omega$.  Any $\alpha\in\ord$ has a unique Cantor normal form
\begin{equation}\label{eq:CNF}
\alpha=a_d\omega^d+\dots+a_1\omega+a_0
\end{equation}
with $a_n\in\nat$, called its \emph{Cantor coefficients}. The \emph{support} of $\alpha$, denoted $\op{Supp}(\alpha)\sub\nat$, consists of all $i$ for which $a_i\neq 0$. The maximum and minimum of the  support of $\alpha$ are  called respectively its \emph{degree} and \emph{order}. An ordinal is a successor ordinal (=of the form $\alpha+1$ for some ordinal $\alpha$) \iff\ its order is zero. The sum of all $a_i$ is called the \emph{valence} of $\alpha$. The ordering $\leq$ on ordinals corresponds   to the lexicographical ordering on the tuples of Cantor coefficients $(a_d,\dots,a_0)$. Let us say that $\alpha$ is \emph{weaker} than $\beta$, denoted $\alpha \aleq\beta$,  if $a_i\leq b_i$ for all $i$, where, likewise,  the $b_i$ are the Cantor coefficients   of $\beta$. 
Note that $\leq $ extends the partial order $\aleq$ to a total order.

Apart from the usual ordinal sum, we make use of the natural or \emph{shuffle} sum $\alpha\ssum\beta$ given in Cantor normal form as $(a_d+b_d)\omega^d+\dots+a_0+b_0$. Note that the shuffle sum is commutative, and $\alpha+\beta$ will in general be smaller than $\alpha\ssum\beta$. In fact, we showed in \cite[Theorem 7.1]{SchSemAdd} that the shuffle sum is the largest possible ordinal sum one can obtain from writing both ordinals as a sum of smaller ordinals and then rearranging their terms. In particular, $\ord$ is closed under both additions. Moreover, since $\alpha\preceq \beta$ \iff\ there exists $\gamma$ such that $\alpha\ssum\gamma=\beta$, we  may view $(\ord,\ssum,\aleq)$ as a partially ordered commutative semi-group.

\section{Length and semi-additivity}
All rings will be commutative, Noetherian, of finite Krull dimension, and all modules will be finitely generated. Throughout, if not specified otherwise, $R$ denotes   a (finite-dimensional, Noetherian) ring and $M$ denotes some (finitely generated)  $R$-module. By Noetherianity, the collection of submodules of $M$ ordered by \emph{inverse} inclusion is a partial well-order, called the \emph{\sch} of $M$ and denoted $\grass R M$ (or just $\gr M$). In particular, any chain in $\grass R M$ is (equivalent to) an ordinal. The supremum of all possible ordinals arising as a chain in this way is called the \emph{length} of $M$ and is denoted  $\lenmod RM$, or when the base ring is clear, simply by $\len M$. 
Note  that the length of $M$ as an $R$-module is the same as that of an $R/\ann RM$-module, and so we may assume, if necessary, that $M$ is faithful. It follows from the Jordan-H\"older theory that this ordinal length coincides with the usual length for modules of finite length. The length of a ring is that of a module over itself. In other words, $\len R$ is the longest descending chain of ideals in $R$. 

 It is useful to have also a transfinite definition of length: we define a \emph\hdim\ $\hd\cdot$ on $\grass R M$, as follows. Put $\hd M:=0$. Given a submodule $N\sub M$, at a successor stage, we say that $\hd N\geq \alpha+1$, if there exists a submodule $N'\sub M$ containing $N$ such that $\hd {N'}\geq \alpha$. If $\lambda$ is a limit ordinal (that is to say, not a successor ordinal), then we say that $\hd N\geq\lambda$, provided for each $\alpha<\lambda$, there exists a submodule $N_\alpha\sub M$ containing $N$ with $\hd {N_\alpha}\geq \alpha$. Finally, we say that $\hd N=\alpha$ if $\hd N\geq\alpha$ but not $\hd N\geq \alpha+1$. We prove in \cite[Theorem 3.10]{SchSemAdd} that   the \hdim\ of the zero module is the length of $M$. In fact, more generally,  for any submodule $N\sub M$, its \hdim\ equals its \emph{co-length}, that is to say,
  \begin{equation}\label{eq:quot}
 \hd N=\len  {M/N}.
\end{equation}
Note that height rank satisfies the following \emph{continuity} property: $\hd N$ is
 less than or equal to the supremum of all $\hd W+1$,
 for all $W$ strictly containing  $N$. 
Using semi-additivity (see introduction) we showed in \cite{SchSemAdd}:

\begin{theorem}[Dimension]\label{T:dim}
Let $M$ be a   finitely generated module over a finite-dimensional Noetherian ring $R$. Then the degree of $\len M$ is equal to the dimension of $M$. In particular,   $R$ is a $d$-dimensional domain \iff\ $\len R=\omega^d$.\qed
\end{theorem}

The \emph{order} of a module is by definition the order of its length, and will be denoted $\order {}M$; the \emph{valence} $\val M$ is the valence of its length. We will calculate these two invariants in \Cor{C:order} below. 
Let us call an exact sequence \emph{strongly \add} if we have equality at both sides of  \eqref{eq:lensemadd}; if we only have equality at the right, we call the sequence \emph\add.  Being strongly \add\ is really a property of ordinals:   $\alpha+\beta=\alpha\ssum\beta$ \iff\ the degree of $\beta$ is at most the order of $\alpha$, and hence

\begin{corollary}\label{C:orddim}
An exact sequence of finitely generated $R$-modules 
$$\Exactseq NMQ,$$
   is strongly \add\ \iff\  $\dim N\leq \order {}Q$.\qed
\end{corollary}

\section{Length as a cohomological rank}
As customary, we define the \emph{dimension} of a prime ideal $\pr\sub R$, denoted $\op{dim}(\pr)$,   as the Krull dimension of the residue ring $R/\pr$. 
We denote the collection of all associated primes (= prime ideals of the form $\ann {}a$ with $a\in M$) by $\op{Ass}(M)$; it is always a finite set.  
We will make frequent use, for a short exact sequence $\Exactseq NMQ$, of the following two inclusions (see, for instance, \cite[Lemma 3.6]{Eis})
 
\begin{align}
 \label{eq:ass1} \op{Ass}(N)&\sub\op{Ass}(M);\\
 \label{eq:ass2} \op{Ass}(M)&\sub\op{Ass}(N)\cup\op{Ass}(Q)
\end{align}

 We define the the (zero-th) \emph{local cohomology}\footnote{This deviates slightly from the practice in \cite[App.~4]{Eis} as we also localize.} of $M$ at $\pr$, denoted $\lc M\pr$, as the $\pr$-torsion of $M_\pr$, that is to say, the submodule of $M_\pr$ consisting of  all elements   that are killed by some power of $\pr$. As $\lc M\pr$ is a module of finite length over $R_\pr$, we denote this length by $\lcl \pr M$ and call it the \emph{local multiplicity of $M$ at $\pr$} (see, for instance, \cite[p.~102]{Eis}). An alternative formulation is through the notion of the   \emph{finitistic  length}  of  a module $M$, defined  as the supremum $\finlen M$ of all $\len N$ with $N\sub M$ and $\len N<\omega$. By Noetherianity,  $M$ has a largest submodule $H$ of finite length, and hence $\finlen M=\len H$. With this notation, we have
 \begin{equation}\label{eq:lcl}
\lcl \pr M =\finlen{M_\pr},
\end{equation} 
for any prime ideal $\pr$, and this is non-zero \iff\ $\pr$ is an associated prime of $M$. 
 We now define
the \emph{cohomological rank} of a module $M$ as  
$$
\cohrk M:=\Ssum_{\pr} \lcl \pr M\omega^{\op{dim}(\pr)}.
$$

It is instructive to view this from the point of view of Chow cycles. 
Let $\mathcal A(R)$ be the \emph{Chow ring} of $R$, defined as the free Abelian group on $\op{Spec}(R)$. An element $D$ of $\mathcal A(R)$ will be called a \emph{cycle}, and  will be represented as a finite sum  $\sum a_i[\pr_i]$, where $[\pr]$ is the symbol denoting the free generator corresponding to the prime ideal $\pr$. The sum of all $a_i$ is called the degree $\op{deg}(D)$ of the cycle $D$. We define a partial order on $\mathcal A(R)$ by the rule that $D\aleq E$, if $a_i\leq b_i$, for all $i$, where $E=\sum b_i[\pr_i]$. In particular, denoting the zero cycle simply by $0$, we call a cycle $D$ \emph{effective} , if $0\aleq D$, and we let $\mathcal A^+(R)$ be the semi-group of effective cycles. This allows us to define a map  from effective cycles to ordinals by sending the effective cycle $D=\sum_ia_i[\pr_i]$ to the ordinal
$$
\binord D:=\Ssum_i a_i\omega^{\op{dim}(\pr_i)}.
$$
Clearly, if $D$ and $E$ are effective, then $\binord{D+E}=\binord D\ssum\binord E$. Moreover, if $D\aleq E$, then $\binord D\aleq\binord E$, so that we get a map $(\mathcal A^+(R),+,\aleq)\to (\ord,\ssum,\aleq)$ of partially ordered semi-groups.

To any $R$-module $M$, we can  assign its \emph{fundamental cycle}, by the rule
\begin{equation}\label{eq:}
\fcyc RM:=\sum_{\pr}\lcl \pr M[\pr]
\end{equation} 
 Immediately  from \eqref{eq:lcl} we get $\cd {} M =\cohrk M$.  
 Our main result now links this cohomological invariant to our combinatorial length invariant :
 
 \begin{theorem}\label{T:cohrk}
For any finitely generated module $M$ over a finite-dimensional Noetherian ring $R$, we   have $\len M=\cd {} M =\cohrk M$.
\end{theorem}


 Before we give the proof, we derive two lemmas. It is important to notice that the first of these is not true at the level of cycles.

\begin{lemma}\label{L:lcord}
If $M\to Q$ is a proper surjective morphism of $R$-modules, then 
$$
\cohrk Q< \cohrk M.
$$
\end{lemma}
\begin{proof}
Let $N$ be the (non-zero) kernel of $M\to Q$, and let $d$ be its dimension. If $\pr\in\op{Ass}(M)$ but not in the support of $N$, then  $M_\pr\iso Q_\pr$, so that they have the same local cohomology. This holds in particular for any $\pr\in\op{Ass}(M)$ with $\op{dim}(\pr)>d$, showing that $\cohrk Q$ and $\cohrk M$ can only start differing at a coefficient of $\omega^i$ for $i\leq d$.  So let $\pr\in\op{Ass}(M)\cap\op{Supp}(N)$ have dimension $d$. In general, local cohomology is only left exact, but by \Lem{L:loccohgen} below, we have in fact an exact sequence \eqref{eq:lc}. 
  Since $\lcl \pr N\neq0$, we must therefore have $\lcl \pr Q<\lcl \pr M$.  It now easily follows that $\cohrk Q< \cohrk  M$.
\end{proof}

\begin{lemma}\label{L:loccohgen}
Given an exact sequence $\Exactseq NMQ$, if $\pr$ is a minimal prime of $M$, then 
\begin{equation}\label{eq:lc}
\Exactseq{ \lc N\pr} { \lc M\pr}{ \lc Q\pr}
\end{equation}
is exact.
\end{lemma}
\begin{proof}
It is well-known (see \cite[App.~4]{Eis}) that $\lc\cdot\pr$ is left exact, so that we only need to prove exactness at the final map. 
By assumption, $M_\pr$ has finite length and hence $N_\pr=\lc N\pr$. Choose $n$ high enough so that    $\pr^nN_\pr=0$. 
Suppose $\bar a\in\lc Q\pr$, that is to say, $\pr^m\bar a=0$ in $Q_\pr$, for some $m$. Let $a\in M$ be a pre-image of $\bar a$ under the surjection $M\to Q$.   Therefore, $\pr^ma\in N_\pr$, whence $\pr^{m+n}a=0$ in $M_\pr$, showing that $a\in\lc M\pr$.
\end{proof}

\begin{corollary}\label{C:semiaddcyc}
Given an exact sequence $\Exactseq NMQ$, if $M$ has no embedded primes, then we have an equality of cycles
\begin{equation}\label{eq:semiaddcyc}
\fcyc RM+D=\fcyc RN+\fcyc RQ
\end{equation} 
where $D$ is an effective cycle supported on $\op{Ass}(Q)\setminus\op{Ass}(M)$. 
\end{corollary} 
\begin{proof}
Let $D$ be the cycle given by \eqref{eq:semiaddcyc}, so that $D$ has support in $\op{Ass}(M)\cup\op{Ass}(Q)$ by \eqref{eq:ass2}. We need to show that $D$ is effective and supported on $\op{Ass}(Q)\setminus\op{Ass}(M)$.
Since any associated prime $\pr$ of $M$ is minimal, $D$ is not supported in $\pr$ by \Lem{L:loccohgen}. On the other hand, any associated prime of $Q$ not in $\op{Ass}(M)$ appears with a positive coefficient in $D$, showing that the latter is   effective.
\end{proof}

\subsection*{Proof of \Thm{T:cohrk}}
Let us first  prove $ \len M\leq  \cohrk M$  by transfinite induction on $\cohrk M$, where the case $\cohrk M=0$ corresponds to $M=0$. Let $N$ be any non-zero submodule of $M$. By Lemma~\ref{L:lcord}, we have $\cohrk {M/N}<\cohrk M$, and hence our induction hypothesis applied to $M/N$ yields $\len{M/N}\leq \cohrk {M/N}<\cohrk M$. Since $\len{M/N}=\hd N$ by \eqref{eq:quot}, continuity therefore shows that $\len M=\hd 0$ can be at most $\cohrk M$, as we needed to show. 

To prove the converse inequality, we induct on the length of $M$. Choose an associated prime $\pr$ of $M$ of minimal dimension, say, $\op{dim}(\pr)=e$. By assumption, there exists $m\in M$ such that $\ann Rm=\pr$. Let $H$ be the submodule of $M$ generated by $m$.  Clearly, $\cohrk H=\omega^e$, and so by what we already proved, $\len{R/\pr}\leq \omega^e$. By \Thm{T:dim}, this then is an equality.  So we may assume that $Q:=M/H$ is non-zero.  By \Lem{L:loccohgen}, we get $\lcl \pr M=\lcl \pr Q-1$. 
By semi-additivity, we have an inequality
$$
\len {Q}+\len H\leq \len M
$$
and therefore, by induction
\begin{equation}\label{eq:IHcd}
\cohrk{Q}+\omega^e\leq \len M
\end{equation}
Let $\primary$ be any associated prime of $M$ different from  $\pr$. By minimality of dimension, $\primary$  cannot contain $\pr$. In particular, $M_\primary\iso Q_\primary$, whence $\lcl  \primary M=\lcl  \primary Q$. Expanding $ \cohrk Q$ in its Cantor normal form $\sum b_i\omega^i$, let $\lambda:=\sum_{i\geq e}b_i\omega^i$. Putting together what we proved so far, we can find an ordinal $\alpha$ with $\low\alpha\geq e$ (stemming from primes associated to $Q$ but not to $M$), such that $\lambda\ssum \omega^e=\cohrk M\ssum\alpha$. Since   $\cohrk Q+\omega^e=\lambda\ssum \omega^e$, we get, from \eqref{eq:IHcd} and the first part,  inequalities
$$
\cohrk M\ssum \alpha\leq \len M\leq \cohrk M
$$
which forces $\alpha=0$ and  all inequalities to be equalities.\qed

\begin{corollary}\label{C:order}
The order of a  module is the smallest dimension of an associated prime, and its valence   is  the   degree of its fundamental cycle. \qed
\end{corollary} 

 By \cite[Proposition 1.2.13]{BH}, over a local ring, we have 
\begin{equation}\label{eq:depthorder}
\op{depth}(M)\leq \order {}M.
\end{equation} 
This inequality can be strict: for example, a two-dimensional domain which is not \CM, has depth one but order two by \Thm{T:dim}.    Our next result gives a constraint on the possible length (not to be confused  with its height rank) of a submodule, which was exploited in \cite{SchBinEndo}, to study binary modules.

\begin{theorem}\label{T:submod}
If $N\sub M$, then $\len N\preceq\len M$. Conversely, if $\nu\aleq\len M$, then there exists a submodule $N\sub M$ of length $\nu$. 
\end{theorem}
\begin{proof}
The first assertion is immediate from \Thm{T:cohrk}, inclusion \eqref{eq:ass1}, and the fact that \eqref{eq:lc} is always left exact.
For the second assertion, let $\mu:=\len M$. 
 We induct on the (finite collection)  of ordinals $\nu$ weaker than $\mu$ to show that there exists a submodule of that length. The case $\nu=0$ being trivial, we may assume $\nu\neq 0$.
Let $i$ be the order of $\nu$ and write $\nu=\theta\ssum\omega^i$ for some $\theta\aleq\nu$. Since then $\theta\aleq\mu$, there exists a submodule of length $\theta$ by induction. Let  $H\sub M$ be   maximal among all submodules of   length $\theta$. By \Thm{T:cohrk}, there exists an $i$-dimensional associated prime $\pr$ of   $M$, such that $\lc H\pr$ is strictly contained in $\lc M\pr$. Hence we can find $x\in M$ outside $H$ such that $\pr x\sub H$. Let $N:=H+Rx$ and let $\bar x$ be the image of $x$ modulo $H$, so that  $N/H=R\bar x$. Since $\pr\bar x=0$, the length of $R\bar x$ is at most $\omega^i$. By semi-additivity applied to the inclusion $H\sub N$, we have an inequality $\len N\leq \theta\ssum\ \len{R\bar x}$, and hence $\len N\leq \nu$. Maximality of $H$ yields  $\theta<\len N$. On the other hand, since   $\len N\preceq\mu$ by our first assertion, minimality of $i$ then  forces $\len N=\nu$, as we needed to show. 
\end{proof}  

\begin{remark}\label{R:submod}
In fact, if $N\sub M$, then $\fcyc {}N\aleq\fcyc {}M$, so that the   fundamental cycle map  is a morphism $\gr M^\circ\to \mathcal A(R)$ of partially ordered sets, where $\gr M^\circ$ is the opposite order given by inclusion.  On the other hand, by \eqref{eq:quot} and \Thm{T:cohrk}, the map $\gr M\to \mathcal A(R)$ given by $N\mapsto \fcyc{}{M/N}$ factors through the length map $\gr M\to \ord$, but there is no natural ordering on $\mathcal A(R)$ for which this becomes a map of ordered sets.
\end{remark} 

We may improve the lower semi-additivity by replacing $\leq $ by $\aleq$:

\begin{corollary}\label{C:semadd}
If $\Exactseq NMQ$ is exact, then   $  \len Q+\len N\aleq \len M$.
\end{corollary} 
\begin{proof}
This is really just a fact about ordinals: with $\mu,\nu,\theta$ being the respective lengths of $M,N,Q$, semi-additivity gives $\theta+\nu\leq\mu\leq\theta\ssum\nu$, whereas \Thm{T:submod} gives $\nu\aleq\mu$, and we now show that these inequalities imply that $\theta+\nu\aleq\mu$.  Write $\mu=\nu\ssum\alpha$ and let $d$ be the dimension of $\nu$. 
For an arbitrary ordinal $\beta$,  we have a unique decomposition $\beta=\beta^+\ssum\beta^-$ with $\beta^-$ of degree strictly less than $d$ and $\beta^+$ of order at least $d$. By assumption, $\nu^+=a\omega^d$ for some $a$, and hence the semi-additivity inequalities at  degree $d$ and higher become $\theta^+\ssum  a\omega^d\leq \alpha^+\ssum a\omega^d\leq\theta^+\ssum a\omega^d$, showing that $\theta^+=\alpha^+$. 
By definition of ordinal sum, $\theta+\nu=\theta^+\ssum\nu$ and so 
$$
(\theta+\nu)\ssum\alpha^-=\theta^+\ssum\nu\ssum\alpha^-=\alpha^+\ssum\alpha^-\ssum\nu=\alpha\ssum\nu=\mu
$$
proving that $\theta+\nu\aleq\mu$. 
\end{proof} 

Recall that a module $M$ is  called \emph{unmixed}  of dimension $d$ if all its associated primes have dimension $d$.  \Thm{T:cohrk} then shows that   $M$ is unmixed 
 \iff\ $\order {}M=\dim M$,    in which case $\len M=\genlen{}M\omega^d$, where we define the \emph{generic length} of $M$ as the sum of the local multiplicities at all $d$-dimensional  primes. \Cor{C:orddim} yields:

\begin{corollary}\label{C:unmadd}
Let $R$ be a $d$-dimensional Noetherian ring and   $\Exactseq NMQ$  a short exact sequence.   If $Q$ is unmixed of dimension $d$, then this sequence is strongly \add.\qed
\end{corollary} 

Recall the \emph{dimension filtration} $\fl D_0(M)\sub \fl D_1(M)\sub\dots\sub \fl D_d(M)=M$ of a $d$-dimensional finitely generated $R$-module $M$ defined by Schenzel in \cite{ScheDimFil}, where $\fl D_i(M)$ is the submodule of  all elements of dimension at most $i$, where we define the \emph{dimension} of an element $x\in M$ as the dimension of the module it generates, that is to say,   $\op{dim}(R/\ann Rx)$. Equivalently, $\fl D_i(M)$ is the largest   submodule of $M$ of dimension at most $i$.

\begin{proposition}\label{P:dimfil}
Given a  $d$-dimensional  module $M$, the exact sequence 
$$\Exactseq{\fl D_i(M)}M{M/\fl D_i(M)}$$
is  {strongly \add}, for each $i$. Moreover, $\len{\fl D_i(M)}$ is obtained from $\len M$ by omitting the monomials of degree bigger than $i$.
\end{proposition} 
\begin{proof}
It follows from  \cite[Corollary 2.3]{ScheDimFil} that the associated primes of $M/\fl D_i(M)$ are precisely the associated primes of $M$ of dimension strictly larger than $i$. By \Cor{C:order}, this means that $\len{M/\fl D_i(M)}$ has order at least $i+1$. Since $\len{\fl D_i(M)}$ has degree at most $i$ by \Thm{T:dim}, the result follows from \Cor{C:orddim}.
%
\end{proof} 

Another way to formulate this result is  as the following formula  for calculating length 
\begin{equation}\label{eq:dimfil}
\len M=\Ssum_{i=0}^d \len{\fl D_i(M)/\fl D_{i-1}(M)},
\end{equation} 
and each non-zero $\fl D_i(M)/\fl D_{i-1}(M)$ is unmixed of dimension $i$ and of length $a_i\omega^i$, where $a_i$ is its generic length.

%

\begin{example}\label{E:nulen}
Let $R$ be the coordinate ring of a plane with an embedded line inside three dimensional space over $k$ given by the equations $x^2=xy=0$ in the three variables $x,y,z$. Using \Thm{T:cohrk}, one easily calculates that $\len R=\omega^2+\omega$, where the associated primes are $\pr=(x)$ and $\mathfrak q=(x,y)$. The ideals of length $\omega$ are exactly those contained in $\pr$. 
The ideals of length $\omega^2+\omega$ (the \emph{open}  ideals in the terminology from the next section), are precisely those 
%
%
that contain a non-zero multiple of $x$ and a non-zero multiple of $y$ (this follows, for instance,  from \cite[Proposition 4.10]{SchBinEndo}). Finally, the remaining (non-zero) ideals of length $\omega^2$, are those contained in $\mathfrak q$ but disjoint from $\pr$ (note that if $I$ is not contained in $\mathfrak q$, then $IR_\pr=R_\pr$ and $IR_{\mathfrak q}=R_{\mathfrak q}$, so that $I$ must be open). 
\end{example}

\section{Open submodules}
By semi-additivity, $\len N$ is at most $\len M$, and, in fact, $\len N\aleq\len M$ by \Thm{T:submod}.  If $M$ has finite length and $N$ is a proper submodule, then obviously its length must be strictly less, but in the non-Artinian case, nothing excludes this from being an equality. So, we call   $N\sub M$ \emph{open}, if $\len N=\len M$. Immediately   from \Thm{T:cohrk} and \Thm{T:submod}, we get

\begin{corollary}\label{C:openval}
A submodule  $N\sub M$ is open \iff\ $\lc N\pr=\lc M\pr$ for all $\pr$, \iff\ $\fcyc RN=\fcyc RM$, \iff\ $\val N=\val M$. \qed
\end{corollary} 

\begin{remark}\label{R:loccoh}
In particular,   the long exact sequence of local cohomology yields that $N\sub M$ is open \iff\  the canonical morphism $\lc {M/N}\pr\to \op H_\pr^1(N)$ is injective, for every (associated) prime $\pr$ (of $M$).
\end{remark}

\begin{proposition}\label{P:big}
Given an exact sequence $\Exactseq NMQ$, if   $\dim(Q)<\order  {}M$, then $N$ is open.   In particular, any non-zero ideal in a domain is open, and more generally, any ideal in an  unmixed ring containing a parameter is open. 
 \end{proposition} 
\begin{proof}
Let $\nu=\sum a_i\omega^i$, $\mu=\sum b_i\omega^i$, and $\theta=\sum c_i\omega^i$ be the respective lengths of $N$, $M$ and $Q$. By semi-additivity, $\mu\leq\theta\ssum\nu$, whence $b_i\leq a_i+c_i$, for all $i$. By assumption, $c_i=0$ whenever $b_i\neq 0$, so that in fact $b_i\leq a_i$, for all $i$, that is to say, $\mu\leq \nu$. Since the other inequality always holds,  $N$ is open.   To prove the last assertion, let $x$ be a parameter in a $d$-dimensional unmixed ring $R$, so that $R/xR$ has dimension strictly less than $d$. Since $\order{}R=d$ by \Thm{T:cohrk}, our first assertion shows that the ideal $(x)$, and hence any ideal containing $x$, is open.  
\end{proof} 

Let us call a submodule $N\sub M$ \emph{\add}, if $\Exactseq NM{M/N}$ is \add, that is to say, if $\len M=\len N\ssum\len {M/N}$. Hence a direct summand is \add\ by semi-additivity. By \Cor{C:orddim}, any submodule   $N$ such that  $\dim N\leq \order {}{M/N}$, is \add, but the converse need not hold. 
 
\begin{proposition}\label{P:osplitopen}
A maximal (proper) submodule is either \add\ or open. In particular, if $M$ has positive order, then any maximal submodule is open.
\end{proposition} 
\begin{proof}
Let $N\varsubsetneq M$ be maximal, so that $Q:=M/N$ is simple, of length one. Let $\nu$ and $\mu$ be the respective lengths of $N$ and $M$. By semi-additivity, we have $\nu\leq\mu\leq\nu+1$. If the former inequality holds, the submodule is open, and if the latter holds, it is \add. The last assertion now follows from \Thm{T:cohrk}, for if $M$ has positive order, its length is a limit ordinal, and so, by \eqref{eq:quot}, no module of finite co-length can be \add.
\end{proof} 

In the ring case, we can even prove:

\begin{proposition}\label{P:lenmax}
If $(R,\maxim)$ is a non-Artinian  local ring, then $\maxim$ is open.
\end{proposition} 
\begin{proof}
  Let $\nu$ and $\mu$ be the respective lengths of $\maxim$ and $R$. In view of \Prop{P:osplitopen},  to rule out  that $\maxim$ is \add, we may assume that it is an associated prime. 
%
%
  Choose $x\in R$ with $\ann{}x=\maxim$ and put $\bar R:=R/xR$. Since $xR$ has length one,  \Cor{C:orddim} applied to the exact sequences  $\Exactseq{xR}R{\bar R}$ and $\Exactseq{xR}\maxim{\maxim\bar R}$ yields $\mu=\len{\bar R}+1$ and $\nu=\len{\maxim\bar R}+1$. By induction, $\len{\maxim\bar R}=\len{\bar R}$, and hence $\nu=\mu$. 
\end{proof} 

\begin{remark}\label{R:lenprim}
As for primary ideals $\mathfrak n$, they will not be open in general if $R$ has depth zero. More precisely, suppose  $\len R=\lambda+n$ with $\lambda$ a limit ordinal and $n\in\nat$. If $\mathfrak n\lc R\maxim  =0$ (which will be the case if $\len{R/\mathfrak n}\geq n$), then $\len {\mathfrak n}=\lambda$. Indeed, the case $n=0$ is trivial, and we may always reduce to this since $\mathfrak n$ is   a module over $R/\lc \maxim R$, and the latter has length $\lambda$ by \Prop{P:big}. 
\end{remark}

\begin{corollary}\label{C:openemb}
If $N\sub M$ is open, then  $\op{Supp}(M/N)$ is nowhere dense in $\op{Supp}(M)$.  The converse   holds if $M$ has no embedded primes. 
\end{corollary} 
\begin{proof}
Let $Q:=M/N$. The condition on the supports means that no minimal prime $\pr$ of $M$  lies in the support of $Q$. However, for such $\pr$, we have $M_\pr=\lc M\pr$.   Hence $N_\pr=M_\pr$ by \Cor{C:openval},  showing that $Q_\pr=0$. Suppose next that $M$ has no embedded primes and $\op{Supp}(Q)$ is nowhere dense in $\op{Supp}(M)$. Let $D:=\fcyc {}N+\fcyc{}Q-\fcyc {}M$ as given by   \Cor{C:semiaddcyc}, and let $\pr$ be an associated prime of $Q$. Since $\op{Supp}(Q)$ is nowhere dense, $\pr$ is not an associated prime of $M$. Since the cycle $\fcyc {}M-\fcyc {}N$ has   support in $\op{Ass}(M)$ and is equal to $ \fcyc {}Q-D$, it must be the zero cycle, since $\fcyc {}Q$ and $D$ have support disjoint from $\op{Ass}(M)$. Hence $N\sub M$ is open by \Cor{C:openval}.
\end{proof} 

Together with \Prop{P:big}, this proves:

\begin{corollary}\label{C:nillow}
If $N\sub M$ is open then   $\op{dim}(M/N)<\op{dim}(M)$, and the converse holds if $M$ is  unmixed.\qed
\end{corollary}

In \cite{SchBinEndo}, we called a module $M$ \emph\topc, if $M$ embeds (as a submodule) in every open submodule. The following generalizes \cite[Corollary 6.12]{SchBinEndo}:

\begin{corollary}\label{C:topc}
If $M$ has no embedded prime ideals, then it is \topc.
\end{corollary} 
\begin{proof}
Let $N\sub M$ be open and put $Q:=M/N$. By \Cor{C:openemb}, no minimal (whence associated) prime of $M$ lies in the support of $Q$. Hence there exists $x\in \ann{}Q$ outside any associated prime of $M$. In particular, $xM\sub N$ and multiplication by $x$ gives an isomorphism $M\iso xM$.
\end{proof} 

\section{The canonical topology}

As the name indicates, there is an underlying topology. To prove this, we need:

\begin{theorem}\label{T:cantop}
The inverse image of an open submodule under a morphism is again open. In particular, if $U\sub  M$ is open, and $N\sub M$ is arbitrary, then $N\cap U$ is open in $N$.
\end{theorem} 
\begin{proof}
We start with proving the second assertion. Let $\mu:=\len M$, $\nu:=\len N$ and $\alpha:=\len{N\cap U}$. We have an exact sequence
$$
\Exactseq{N\cap U}{N\oplus U}{N+U}.
$$
Since $N+U$ is again open, semi-additivity yields $\len{N\oplus U}\leq \alpha\ssum\mu$. Since the former is equal to $\nu\ssum\mu$, we get $\nu\leq \alpha$, showing that $N\cap U$ is open in $N$. To prove the first assertion, let $f\colon M\to N$ be an arbitrary \homo\ and let $V\sub N$ be an open submodule. Let $G\sub M\oplus N$ be the graph of $f$ and let $p\colon G\to M$ be the projection onto the first coordinate. Since $M\oplus V$ is open in $M\oplus N$ by semi-additivity, $(M\oplus V)\cap G$ is open in $G$, by what we just proved. Since $p$ is an isomorphism, the image of $(M\oplus V)\cap G$ under $p$ is therefore open  in $M$. But this image is just $\inverse fV$, and so we are done.
\end{proof}

We can now define a topology on $M$ by letting the collection of open submodules be a basis of open neighborhoods of $0\in M$. This is indeed a basis since the intersection of finitely many opens is again open by \Thm{T:cantop}.  An arbitrary open in this topology is then a  (possibly infinite) union of cosets $x+U$ with $U\sub M$ open and $x\in M$. If a submodule $N$ is a union of cosets $x_i+U_i$ of open submodules $U_i\sub M$, then one such coset,  $x+U$ say, must contain $0$, so that $U\sub N$, and hence the inequalities $\len U\leq\len N\leq \len M$ are all equalities, showing that $N$ is indeed open in the previous sense. 

We call this the \emph{canonical topology} on $M$, and \Thm{T:cantop} shows that any \homo\ is continuous in the canonical topology. Moreover,  multiplication on any ring is continuous: given $a_1,a_2\in R$ and an open ideal $I$ such that $a_1a_2\in I$, let $J_i:=a_iR+ I$. Hence $J_i$ is an open neighborhood of $a_i$ and $J_1\cdot J_2\sub a_1a_2R+I=I$. If $M$ has dimension zero, then the canonical topology is trivial, since $M$ is then the only open submodule. 
The complement of an open module $N$ is the union of all cosets $a+N$ with $a\notin N$, and hence is also open.  In particular, an open module is also closed, and the quotient topology on $M/N$ is discrete, whence in general  different from the canonical topology. The zero module is closed \iff\ the intersection of all open submodules is $0$, that is to say, \iff\ the canonical topology is Haussdorf. 
%
%
\begin{corollary}\label{C:nonArt}
A module is non-Artinian \iff\ its canonical topology is non-trivial.
\end{corollary} 
\begin{proof}
One direction is immediate since an Artinian module has no proper open submodules. For the converse, we show, by induction on  $\len M$, that $M$ has a non-trivial, open submodule.   Assume first that $\mu$ is a limit ordinal. Choose a  submodule $N$ of $M$ of \hdim\ $1$. By \eqref{eq:quot}, this means $\len{M/N}=1$, and so $N$ is open by \Prop{P:big}. 
%
%
Next, assume $\mu=\nu+1$. Let $H$ be a submodule of \hdim\ $\nu$, so that by \eqref{eq:quot} again, $\bar M:=M/H$ has length $\nu$. By induction, we can find a proper open submodule of $\bar M$, that is to say, we can find $N\varsubsetneq M$ containing $H$ such that $\len{N/H}=\nu$. Semi-additivity applied to the inclusion $H\sub N$ yields $\nu+\len H\leq\len N$. Since $\len H\neq 0$ and $\len N\leq \nu+1$, we get equality, that is to say, $\len N=\nu+1$, whence $N$ is open.
\end{proof}

\begin{proposition}\label{P:clzero}
The closure of a submodule   $N\sub M$ is equal to $N+\fl D_0(M)$. In particular, a submodule is closed \iff\ it has the same finitistic length as $M$.
\end{proposition} 
\begin{proof}
Let $H:=\fl D_0(M)$, and let $W$ be an open submodule containing $N$. By \Thm{T:cantop}, the intersection $H\cap W$ is open in $H$, and since $H$ has finite length (so that its topology is trivial), we must have $H=W\cap H$, proving that $H$ lies in $W$. As this holds for all opens  $W$ containing $N$, the closure of $N$ contains $H$.

Using \Cor{C:orddim}, one easily shows that the quotient topology on $M/H$ is equal to the canonical topology. Therefore, to calculate the closure, we may divide out $H$, assume that $M$ has positive order, and we then need to show that $N$ is closed. Let $x\in M$ be any   element not in $N$  and let $\maxim$ be a maximal ideal containing $(N:x)$. Since $M/\maxim^kM$ has finite length, each $\maxim^kM$ is open by \Prop{P:big}, whence so is each $N+\maxim^kM$. If $x$ is contained in each of these, then it is contained in their intersection $W:=\cap(N+\maxim^kM)$. By Krull's Intersection theorem, there exists $a\in \maxim$ such that $(1+a)W\sub N$. In particular, $1+a\in(N:x)\sub \maxim$, contradiction. So $x$ lies outside some open $N+\maxim^kM$, and hence does not belong to the closure of $N$.

The last assertion is now also clear by \Thm{T:cohrk}, since the finitistic length of $M$ is equal to the length of $\fl D_0(M)$.
\end{proof} 

\begin{corollary}\label{C:hauss}
    A module has a separated canonical topology \iff\ its order is positive \iff\ any submodule is closed.\qed
\end{corollary} 

\begin{example}\label{E:adictop}
If $(R,\maxim)$ is   local and $M$ has positive depth, then the canonical topology refines the $\maxim$-adic topology on $M$: indeed $\maxim^k M$ is then open by \Prop{P:big}. We already showed that this is no longer true in rings of positive depth in \Rem{R:lenprim}. 
Also note that the canonical topology on a local domain is strictly finer than its adic topology, since all non-zero ideals are open.   

The ring $\pow k{x,y}/(x^2,xy)$, of length $\omega+1$, is not Haussdorf as the closure of the zero ideal is the ideal $(x)$ by \Prop{P:clzero}. It is the only closed, non-open ideal, since the closure of any ideal must contain $x$ whence is open when different from $(x)$. In particular, whereas the canonical topology is not Haussdorf,  the adic one is.
\end{example} 

Recall that a submodule $N\sub M$ is called \emph{essential} (or \emph{large}), if it intersects any non-zero submodule non-trivially.

\begin{corollary}\label{C:bigbg}
An open submodule is {essential}. In particular,  $0$ is a limit point of any non-zero submodule.
\end{corollary} 
\begin{proof}
Let $N\sub M$   have the same length and let $H$ be an arbitrary  submodule. Suppose $H\cap N=0$, so that $H\oplus N$ embeds  as a submodule of $M$. In particular, $\len H\ssum\len N\leq \len M$ by semi-additivity, forcing $\len H$, whence $H$, to be zero.

To prove the second assertion, let $N$ be non-zero and let $U$ be an open containing $0$. Since $U$ is the union of cosets of open submodules and contains $0$, it must contain at least one open submodule $W$. Since $W$ is essential by our first assertion, $W\cap N\neq0$.
\end{proof} 

The converse is false: in an Artinian local ring, the socle is essential, but it is clearly not open.  A less trivial example is given by the ideal $\pr:=(x,y)$ in the ring $R=\pol k{x,y,z}/\pr^2$, which is essential but not open. Indeed, $\fcyc {}R=3[\pr]$ and hence $R$ has length $3\omega$ by \Thm{T:cohrk}. Since $R/\pr$ is a one-dimensional domain, its length is $\omega$ by  \Thm{C:orddim}, and hence $\len\pr=2\omega$ by \Cor{C:orddim}. To see that $\pr$ is essential, we can use the following proposition with $S=\pol k{x,y}/\pr^2$ (so that $R=\pol Sz$). 
 
\begin{proposition}\label{P:essff}
Let $(S,\pr)$ be a local ring and $S\to R$ a   flat extension. Then $\pr R$ is an essential ideal of $R$.
\end{proposition} 
\begin{proof}
If $\pr R$ were not essential, we could find a non-zero $x\in R$ such that $\pr R\cap xR=0$. In particular, $x\pr=0$. By  flatness, $x\in\ann S\pr R\sub\pr R$, contradiction.
\end{proof}

\begin{theorem}\label{T:indtop}
If $M$ has no embedded primes, then the induced topology on a submodule $N\sub M$ is the same as the canonical topology on $N$.
\end{theorem} 
\begin{proof}
In view of \Thm{T:cantop}, we only have to show that if $W\sub N$ is open, then there exists $U\sub M$ open such that $W=U\cap N$. Let $U$ be maximal such that $U\cap N=W$. Suppose $U$ is not open, so that there exists an associated prime $\pr$ of $M$ with $\lc U\pr\varsubsetneq \lc M\pr$ by \Cor{C:openval}. In particular, we can find $x\in M\setminus U$ with $s\pr x\sub U$, for some $s\notin\pr$. By maximality, $(U+Rx)\cap N$ must contain an element $n$ not in $W$. Write $n=u+rx$ with $u\in U$ and $r\in R$. In particular, $s\pr n=s\pr u$ lies in $U\cap N=W$. In other words, we showed that $\lc{N/W}\pr\neq 0$. In particular, $N_\pr\neq 0$, so that $\pr$ is a minimal prime of $N$. By \Lem{L:loccohgen}, we have $\lcl \pr W<\lcl \pr W+\lcl \pr {W/N}=\lcl \pr N$, so that $\len W<\len N$ by \Thm{T:cohrk}, contradiction.
\end{proof} 

I do not know whether this remains still true if there are embedded primes, but see \cite[Corollary 4.4]{SchCond}.

\section{Higher order topologies}
Given a submodule $N\sub M$, we call any submodule $H$ such that $N\cap H=0$ and $N+H$ is open a \emph{quasi-complement} of $M$. Quasi-complements may not always exist, but they do, for instance, for binary modules (\cite[Proposition 4.12]{SchBinEndo}).  In fact, in \cite[Corollary 4.5]{SchCond}, we show that their existence is equivalent with the module being condense, meaning that every essential submodule is open. 

Given $-1\leq i\leq d$, we call any quasi-complement of $\fl D_i(M)$ an \emph{$i$-open}. If $i<\order{}M$, then $i$-open is the same as open, and, in particular, $(-1)$-open   just means open. Given an ordinal $\mu$   and $i\in\nat$, we will write $\mu^+_i$ for the sum of all terms in the Cantor normal form of $\mu$ of degree at least $i+1$.

\begin{lemma}\label{L:iopen}
Let $M$ be a module of length $\mu$. 
A submodule $N\sub M$ is an $i$-open \iff\ its length is equal to $\mu^+_i$. 
\end{lemma} 
\begin{proof}
 Write $\mu=\mu^+\ssum\mu^-$ with $\mu^+:=\mu^+_i$, and let   $D:=\fl D_i(M)$. By \Prop{P:dimfil}, we have $\len D=\mu^-$. If   $N$ is an $i$-open, then $N\cap D=0$ and $\len{N+D}=\mu$. By additivity, $\mu=\len N\ssum\len D$, and hence $\len N=\mu^+$. Conversely, suppose $\len N=\mu^+$ and put $\alpha:=\len{D\cap N}$. By \Thm{T:submod}, we have  $\alpha\aleq \mu^+$ and $\alpha\aleq\mu^-$, which implies $\alpha=0$. Therefore $N+D\iso N\oplus D$ has length $\mu^+\ssum\mu^-=\mu$, proving that it is open.
\end{proof} 

Together with \Thm{T:submod}, this shows that there are always non-trivial $i$-opens whenever $i<d$, where $d$ is the dimension of $M$. In particular, $(d-1)$-opens are the unmixed submodules  of the same generic length  as $M$. The analogue of \Thm{T:cantop} also holds:

\begin{theorem}\label{T:itop}
The inverse image of an $i$-open   under a morphism is again an $i$-open. In particular, if $N\sub  M$ is an $i$-open, and $W\sub M$ is arbitrary, then $N\cap W$ is an $i$-open in $W$.
\end{theorem} 
\begin{proof}
As in the proof of \Thm{T:cantop}, we only need to verify the latter property. Let $\mu:=\len M$, $\theta:=\len W$, $\alpha:=\len {N\cap W}$ and $\beta:=\len {N+W}$. As before, we decompose any ordinal $\lambda$ as $\lambda^+\ssum\lambda^-$ with $\lambda^+=\lambda_i^+$.  In particular, $\len N=\mu^+$ by \Lem{L:iopen}, and hence  also  $\mu^+=\beta^+$ by  \Thm{T:submod}. Semi-additivity applied to the exact sequence $\Exactseq {N\cap W}{N\oplus W}{N+W}$ gives $\mu^+\ssum\theta \leq \alpha\ssum\beta$.  and so, taking dropping all terms of degree at most $i$ in the above inequality yields $\mu^+\ssum\theta^+\leq \alpha^+\ssum\mu^+$. Since $\alpha\aleq\theta$ by \Thm{T:submod}, we get $\theta^+=\alpha^+$. On the other hand, since $N\cap \fl D_i(W)\sub N\cap \fl D_i(M)=0$, we get $\alpha=\alpha^+$ by \Thm{T:cohrk},   showing that $N\cap W$ is an $i$-open in $W$ by \Lem{L:iopen}.
\end{proof} 

This allows us to define a topology based on $i$-opens: the \emph{$i$-th order topology} on $M$ is defined by using as basic open subsets the $i$-opens and their cosets. By \Thm{T:itop}, any morphism between modules is continuous in the respective $i$-th order topologies on these modules. Unlike the canonical topology (that is to say, say, the $(-1)$-th order topology), a submodule may be open in the $i$-th order topology  without being an $i$-open (for instance, any open submodule is also open in the $i$-th order topology):

\begin{lemma}\label{L:openitop}
A submodule $N\sub M$ is open in the $i$-th order topology \iff\ $\mu^+_i\aleq\len N$, where $\mu:=\len M$.
\end{lemma} 
\begin{proof}
Let $\nu:=\len N$. If $\mu^+_i\aleq\nu$, then $N$ contains a submodule $U$ of length $\mu_i^+$ by \Thm{T:submod}. As $U$ is $i$-open by \Lem{L:iopen}, we see that $N$ is open in the $i$-th order topology, as it is the union of all $a+U$ for $a\in N$.  Conversely, if $N$ is open in this topology, it is a union of basic open subsets. Let $U$ be one such basic open containing $0$. As $U$ is a coset of an $i$-open submodule, it must be equal to it. Hence $\len U=\mu^+_i$ by \Lem{L:iopen}, and by \Thm{T:submod}, this then yields $\mu^+_i\aleq \nu$.
\end{proof}

\begin{proposition}\label{P:powmaxid}
In a local ring $(R,\maxim)$, any sufficiently high power of $\maxim$ is $0$-open. In particular, any primary ideal is open in the $0$-th order topology.
\end{proposition} 
\begin{proof}
If $R$ has positive depth, whence positive order, any power of $\maxim$ is open by \Prop{P:big}, and in this case open is the same as $0$-open. The case that $R$ is Artinian is also immediate. So assume $R$ has depth zero, and write its length as $\lambda+e$ with $\lambda$ of positive order and $e\in\nat$. Let $\mathfrak d:=\fl D_0(R)=\lc R\maxim$. By the Artin-Rees lemma, $\maxim^n\cap\mathfrak d \sub\maxim^{n-c}\mathfrak d$ for some $c$ and all $n\geq c$. Hence, for $n$ sufficiently large, we get $\maxim^n\cap \mathfrak d=0$. Let $\nu:=\len{\maxim^n}$ and let $l$ be the length of $R/\maxim^n$. By semi-additivity, we have $\nu\leq \lambda+e\leq\nu+l$, showing that $\nu=\lambda$ and hence that $\maxim^n$ is $0$-open by \Lem{L:iopen}. The last statement is clear since any primary ideal contains large powers of $\maxim$.
\end{proof}

We may generalize this as follows:

 \begin{proposition}\label{P:orderopen}
Let $e$ be the order of $R$ and let $\id$ be the intersection of all associated primes of $R$ of dimension $e$. Then any sufficiently high power of $\id$ is $e$-open. 
\end{proposition} 
\begin{proof}
Let $\mathfrak d:=\fl D_e(R)$. Any element in $\mathfrak d$ is annihilated by some power of some associated prime ideal  of $R$ of dimension $e$, whence   by some power of $\id$. By the Artin-Rees Lemma, $\id^n\cap \mathfrak d\sub\id^{n-c}\mathfrak d$ for some $c$ and all $n\geq c$. Combining both observations shows that $\id^n\cap \mathfrak d=0$ for $n\gg0$. Write $\len R=\lambda+r\omega^e$, with $\lambda$ of order at least $e+1$ and $r\in\nat$. Let $\alpha$ and $\rho$ be the respective lengths of $\id^n$ and $R/\id^n$. Since the latter has dimension $e$, we get $\rho^+_e=0$, and since the former does not contain any element of dimension $e$, we get $\alpha^+_e=\alpha$ by \Thm{T:cohrk}. On the other hand, semi-additivity yields  $ \rho+\alpha\leq\lambda+r\omega^e\leq \rho\ssum\alpha$. Considering only terms in degree at least $e+1$ yields $\alpha=\lambda$, showing that $\id^n$ is an $e$-open by  \Lem{L:iopen}.
\end{proof} 

If we only want to ensure openness in the $i$-th order topology, we can use
 
\begin{corollary}\label{C:openithtop}
If $N\sub M$ is a submodule so that $M/N$ has dimension at most $i$, then $N$ is open in the $i$-th order topology.
\end{corollary} 
\begin{proof}
Let $\mu,\nu$ and $\theta$ be the respective lengths of $M,N$ and $M/N$. By semi-additivity, we have $\theta+\nu\leq \mu\leq \theta\ssum\mu$. Since $\theta^+_i=0$ by assumption, taking terms of degree at least $i+1$ gives $\nu^+_i=\mu^+_i$. Hence $N$ is open in the $i$-th order topology by \Lem{L:openitop}.
\end{proof} 

Let us call a submodule $N\sub M$ \emph{conjunctive}, if for every other submodule $H$, we have
\begin{equation}\label{eq:conj}
\len{N\cap H}=\len N\en\len H
\end{equation} 
where the \emph{meet} $\alpha\en\beta$ of two ordinals is defined as their $\aleq$-infimum. If the respective Cantor normal forms are $\alpha=\sum a_i\omega^i$ and $\beta=\sum b_i\omega^i$, then $\alpha\en\beta=\sum \min(a_i,b_i)\omega^i$. Note that by \Thm{T:submod}, the length of $N\cap H$ is weaker than $\len N$ and $\len H$, so that we always have an inequality $\aleq$ in \eqref{eq:conj}. We showed in \cite[Theorem 4.6]{SchBinEndo} that if $M$ has binary length (meaning that the only coefficients in the Cantor normal form of $\len M$ are $0$ and $1$), then any submodule is conjunctive. The following result generalizes this since, by \Lem{L:iopen}, in a module of binary length, any submodule is $e$-open for $e$  its order.

\begin{corollary}\label{C:econj}
For any $i\in\nat$, all $i$-opens are conjunctive.
\end{corollary} 
\begin{proof}
Let $\mu:=\len M$, fix some $i$,  and write $\mu^+$ for $\mu^+_i$, etc. 
Let $U$ be an $i$-open, so that its length is  $\mu^+$ by \Lem{L:iopen}. Let $N\sub M$ be arbitrary, and let $\nu$ and $\theta$ be the respective length of $N$ and $U\cap N$.   By \Thm{T:itop}, the intersection $U\cap N$ is an $i$-open in $N$, so that $\theta=\nu^+$  by \Lem{L:iopen}. Since $\nu\en\mu^+=\nu^+$, the assertion follows. 
\end{proof} 

\section{Degradation}

By \emph{degradation}, we mean the effect that source and target of a morphism have on its kernel. 
We start with a general observation about kernels: given two $R$-modules $M$ and $N$, let us denote the subset of $\grass RM$ consisting of all $\op{ker}(f)$, where $f\in\hom RMN$ runs over all morphisms, by $\mathfrak{ker}_R(M,N)$. 

\begin{theorem}\label{T:kerlen}
As a subset of $\grass RM$, the orded set $\mathfrak{ker}_R(M,N)$ has finite length.
\end{theorem} 
\begin{proof}
Let $f\colon M\to N$ be a morphism, and let  $\theta$ be the  length of its   image.   By \Thm{T:submod}, we have  $\theta\aleq{}\len N$. In particular, there are at most  $2^{\val N}$ possibilities for  $\theta$. I claim that if $g\colon M\to N$ is a second morphism and $\op{ker}(g)\varsubsetneq \op{ker}(f)$, then $\len{\op{Im}(g)}$ is strictly bigger than $\theta$. From this claim it then follows that any chain in $\mathfrak{ker}_R(M,N)$ has length at most $2^{\val N}$. To prove the claim, we have $\hd{\op{ker}(g)}=\len{N/\op{ker}(g)}=\len{\op{Im}(g)}$, by \eqref{eq:quot}. By assumption, $\op{ker}(g)$ is strictly contained in $\op{ker}(f)$, and hence it has strictly bigger \hdim, showing the claim.
\end{proof} 

Viewing the elements in the dual $M^*:=\hom RMR$ as morphisms, we have the following remarkable fact about their kernels.

\begin{corollary}\label{C:uniker}
If $R$ is a domain, then there are no inclusion relations among the kernels of elements in  $M^*$. In particular, for each $K\sub M$, the set $H_K$ of elements in $M^*$ with kernel equal to $K$ together with the zero morphism, is a submodule of $M^*$.
\end{corollary}
\begin{proof}
By \Thm{T:dim}, the length of $R$ is $\omega^d$ with $d=\op{dim}(R)$, so that its valence  is one. Hence, by the  proof of \Thm{T:kerlen}, there are only two possibilities for the \hdim\ of a kernel: zero, corresponding to the zero element, and $\omega^d$. This proves the first assertion since smaller submodules have strictly larger \hdim. Assume next that $f,g\in H_K$. Since $K$ then lies in the kernel of $rf+sg$, for any $r,s\in R$, the latter kernel is either $K$ or $M$.
\end{proof}

For each $v$, let $\grass vM$ be the subset of the \sch\ $\gr M$ consisting of all submodules for which $M/N$ has valence at most $v$. The same argument shows that each $\grass vM$ has finite length: indeed, the \hdim\ of $N\in\grass vM$ satisfies $\hd N\aleq{}v\omega^d+v\omega^{d-1}+\dots+v$, where $d$ is the dimension of $M$, and therefore, we only have finitely many possibilities for $\hd N$. Note that the union of the $\grass vM$ is $\gr M$, so that the \sch\ can be written as a union of suborders of finite length.  
We can now list some examples of degradation:

\begin{corollary}\label{C:openker}
If $M$ and $N$ have no associated primes in common, then any morphism between them has open kernel.
\end{corollary} 
\begin{proof}
Let $f\colon M\to N$ be a morphism and let $K\sub M$ and $Q\sub N$ be its respective kernel and image, so that we have an exact sequence $\Exactseq KMQ$. For any associated prime $\pr$ of $M$, we have $\lc N\pr=0$ whence also $\lc Q\pr=0$. By left exactness of local cohomology,  $\lc K\pr=\lc M\pr$, and hence $K$ is open by \Cor{C:openval}.
\end{proof} 


\begin{corollary}\label{C:zerohom}
If $\op{dim}(M)<\order {}N$, then $\hom RMN=0$.
\end{corollary} 
\begin{proof}
%
Let $d<e$ be the respective dimension of $M$ and order of $N$, and let $x\in M$. Since $x$ has dimension at most $d$, so does $f(x)$. Hence $f(x)$ must be zero, since $\fl D_{e-1}(N)=0$ by \Prop{P:dimfil}.
\end{proof}

\begin{theorem}\label{T:pingpong} 
Suppose $M$ is torsion-free over $\zet$.
If $M$ and $N$ have no associated primes in common, then there exists $k\in\nat$, such that for any choice of morphisms $f_i\colon M\to N$ and $g_i\colon N\to M$, with $i=\range 1k$, the composition $g_kf_k\cdots g_1f_1=0$.
\end{theorem} 
\begin{proof}
Let $\mathfrak N$ be the subset of all endomorphisms of $M$ whose kernel is open. Let us quickly repeat the argument from \cite{SchBinEndo} that $\mathfrak N$ is a two-sided ideal in $\ndo M$ consisting entirely of nilpotent endomorphisms: indeed, if $a,b\in\mathfrak N$ and $c\in\ndo M$, then the kernel of $a+b$ contains the open $\op{ker}(a)\cap\op{ker}(b)$; similarly, the kernel of $ca$ contains the open $\op{ker}(a)$, and the kernel of $ac$ is the inverse image of $\op{ker}(a)$ under $c$, which is   open by continuity. This shows that $\mathfrak N$ is a two-sided ideal. Finally, some power $a^n$ satisfies $\op{ker}(a^n)=\op{ker}(a^{n+1})$ by Noetherianity, which implies $\op{ker}(a^n)\cap\op{Im}(a^n)=0$. Since $\op{ker}(a^n)$ is a fortiori open, whence essential by \Cor{C:bigbg}, the submodule $\op{Im}(a^n)$ must be zero, showing that $a$ is nilpotent. By   the Nagata-Higman Theorem (\cite{HigNag,NagNil}), the ideal $\mathfrak N$ is therefore nilpotent,   that is to say, $\mathfrak N^k=0$, for some $k$. Since the kernel of each $f_i$ is open  by \Cor{C:openker}, so is the kernel of each $g_if_i$, showing that $g_if_i\in\mathfrak N$, and the claim follows.
\end{proof} 

\begin{corollary}\label{C:dimpingpong}
Suppose $\mathbb Q\sub R$. 
Let $e$ be the maximal dimension of a common associated prime of $M$ and $N$. Then there is some $k$, such that for any choice of morphisms $f_i\colon M\to N$ and $g_i\colon N\to M$, with $i=\range 1k$, the image of the composition $g_kf_k\cdots g_1f_1$ has dimension at most $e$.
\end{corollary} 
\begin{proof}
We take the convention that $e=-1$ if there are no common associated primes and we assign $-1$ to the dimension of the zero module. Hence the assertion is now just \Thm{T:pingpong} in  case $e=-1$. For $e\geq0$, let $M':=M/\fl D_e(M)$ and $N':=N/\fl D_e(N)$. Since $M'$ and $N'$ have no associated primes in common by maximality of $e$, we can find some $k$   as in \Thm{T:pingpong}. Choose $k$ many morphisms $f_i$ and $g_i$ as in the hypothesis, and let $h$ be their composition. Since morphisms cannot increase dimension, they induce morphisms between $M'$ and $N'$, and hence, by choice of $k$, the endomorphism on $M'$ induced by $h$ is zero. It follows that $h(M)\sub\fl D_e(M)$, as claimed.
\end{proof}

\begin{remark}\label{R:torsion}
The torsion restrictions above and below come from our application of the Nagata-Higman Theorem, which requires some form of torsion-freeness (see \cite[Remark 6.5]{SchBinEndo} for a further discussion). One can weaken these assumptions: for instance, in \Cor{C:dimpingpong}, we only need that $\pr\cap \zet=0$, for any associated prime $\pr$ of $M$.  
\end{remark} 

To extend \Thm{T:pingpong} to  more modules, let us say that an endomorphism $f\in\ndo M$    \emph{reflects through} a collection of modules $\mathcal N$, if for each $N\in\mathcal N$, we can factor $f$ as $M\to N\to M$. We can now prove:

\begin{theorem}\label{T:multipingpong}
Let $M$ be a module without $\zet$-torsion, and let $\mathcal N$ be  a collection of modules such that there is no prime ideal which is associated to $M$ and to every $N\in\mathcal N$. Then there exists $k\in\nat$, so that any product of $k$-many endomorphisms reflecting through $\mathcal N$ is zero.
\end{theorem}  
\begin{proof}
As in the proof of \Thm{T:pingpong}, it suffices  to show that any  endomorphism $f\in\ndo M$ reflecting to $\mathcal N$  has open kernel. Let $K$ be its kernel and let $\pr$ be an associated prime of $M$. By assumption, there exists $N\in\mathcal N$ such that $\pr$ is not an associated prime of $N$. By definition, there  exists a factorization $f=hg$ with $g\colon M\to N$ and $h\colon N\to M$. Let $H$ be the kernel of $g$. By the argument in the proof of \Cor{C:openker} applied to $g$, we get $\lc {H}\pr=\lc M\pr$. Since $H\sub K$, this implies $\lc K\pr=\lc M\pr$, and since this holds for all associated primes $\pr$ of $M$, we proved   by \Cor{C:openval} that $K$ is open. 
\end{proof} 

If, instead, there are common associated primes, let $e$ be the maximum of their dimensions.  By the same argument as in the proof of \Cor{C:dimpingpong}, we may then conclude that the image of any product of $k$-many  endomorphisms reflecting through $\mathcal N$  has dimension at most $e$, provided every associated prime lies above $(0)\sub \zet$ (see \Rem{R:torsion}).

%

%
%
\providecommand{\bysame}{\leavevmode\hbox to3em{\hrulefill}\thinspace}
\providecommand{\MR}{\relax\ifhmode\unskip\space\fi MR }
\providecommand{\MRhref}[2]{%
  \href{http://www.ams.org/mathscinet-getitem?mr=#1}{#2}
}
\providecommand{\href}[2]{#2}

\end{document}